\documentclass[10pt]{amsart}
\usepackage{amsrefs}

\newtheorem{theorem}{Theorem}

\newtheorem{example}[theorem]{Example}

\newtheorem{definition}[theorem]{Definition}
\newtheorem{lemma}[theorem]{Lemma}
\newtheorem{remark}[theorem]{Remark}
\newtheorem{proposition}[theorem]{Proposition}
\newtheorem*{equi3}{An equivalent version of the ISP}

\begin{document}

\title[On the invariant subspace problem]
 {On the invariant subspace problem via universal Toeplitz operators on the Hardy space $H^{2}(\mathbb{D}^{2})$}

\author[João Marcos R. do Carmo]{João Marcos R. do Carmo}
\author[Marcos S. Ferreira]{Marcos S. Ferreira}



\date{September 06, 2023}

\begin{abstract}
The \textit{Invariant Subspace Problem} (ISP) for Hilbert spaces asks if every bounded linear operator has a non-trivial closed invariant subspace. Due to the existence of universal operators (in the sense of Rota) the ISP can be solved by proving that every minimal invariant subspace of a universal operator is one dimensional. In this paper, we obtain a nontrivial invariant subspace of $T^{*}_{\varphi}|_{M}$, where $T_{\varphi}$ is the Toeplitz operator on the Hardy space over the bidisk $H^{2}(\mathbb{D}^{2})$ induced by the symbol $\varphi\in H^{\infty}(\mathbb{D})$ and $M$ is a $T_{\varphi}^{*}$-invariant subspace. We use this fact to get sufficient conditions for the ISP.
\end{abstract}

\maketitle

\section{Introduction}

The \textit{Invariant Subspace Problem} (ISP) is one of the most important problems in functional analysis that remains unsolved for separable, infinite dimensional Hilbert spaces, which asks: given a Hilbert space $\mathcal{H}$ and a bounded linear operator $T$ on $\mathcal{H}$, does $T$ have a nontrivial invariant subspace? In recent years, several operator theorists have been developing approaches in an attempt to solve the ISP. Among these, we highlight some research involving Toeplitz operators and composition operators on the Hardy space and similar ones \cite{Cowen, CowenII, CowenIII, CowenIV, Carmo} and the classical book \cite{Radjavi} by Radjavi and Rosenthal and the monograph by Chalendar and Partington \cite{Chalendar}.

In 1960, Rota \cite{Rota} introduced the idea of an operator with an invariant subspace structure so rich as to model every Hilbert space operator. The notion of this operator is what we call today the universal operator.

\begin{definition}\cite[p. 213]{Chalendar}
Let $\mathcal{H}$ be a Hilbert space, U a bounded operator on $\mathcal{H}$ and $\mathcal{L}(\mathcal{H})$ the algebra of bounded operators on $\mathcal{H}$. We say U is universal for $\mathcal{H}$ if for each non zero $A\in\mathcal{L}(\mathcal{H})$ there is an invariant subspace M for U and a non zero number $\lambda$ such that $\lambda A$ is similar to $U|_{M}$, that is, there is a linear isomorphism X of $\mathcal{H}$ onto M such that $UX=\lambda XA$.
\end{definition}

The main tool for obtaining universal operators is the following Caradus criterion.

\begin{theorem}\cite[p. 527]{Caradus}
If $\mathcal{H}$ is a separable Hilbert space and $U\in\mathcal{L}(\mathcal{H})$ such that:
\begin{enumerate}
    \item [1.] The null space of U is infinite dimensional.
    \item [2.] The range of U is $\mathcal{H}$,
\end{enumerate}
then $U$ is universal for $\mathcal{H}$.
\end{theorem}

A well-known example of a universal operator on $\mathcal{H}$ is the adjoint of a unilateral shift of infinite multiplicity, which was introduced by Rota. In fact, considering $S^{*}$ acting on
$$
\ell^{2}(\mathcal{H})=\left\{(f_{n})_{n=0}^{\infty}:\sum_{n=0}^{\infty}||f_{n}||^{2}_{\mathcal{H}}<\infty\right\}
$$
by
$$
S^{*}(f_{0},f_{1},f_{2},\cdots)=(f_{1},f_{2},\cdots)
$$
we have that $S^{*}$ satisfies the Caradus criterion and so is an universal operator.  Following this same idea, we obtain other interesting examples of universal operators, namely the adjoints of Toeplitz operators on the Hardy space over the disk induced by analytic symbols.  In this direction, Cowen and Gallardo-Gutiérrez showed the following.

\begin{theorem}\cite[Theorem 5]{Cowen}
Let $\varphi\in H^{\infty}(\mathbb{D})$ such that $1/\varphi\in L^{\infty}(\mathbb{T})$. If the Toeplitz operator $t_{\varphi}^{*}$ has infinite
dimensional kernel, then $t_{\varphi}^{*}$ is universal for $H^{2}(\mathbb{D})$. 
\end{theorem}

For Toeplitz operators on the Hardy space over the polydisk, Ferreira and Noor showed the following.

\begin{theorem}(\cite[Theorem 1]{Ferreira})\label{t1}
Let $\varphi\in H^{\infty}(\mathbb{D}^{n})$ for $n>1$. Then $T^{*}_{\varphi}$ satisfies the Caradus criterion for universality if, and only if, $\varphi$ is invertible in $L^{\infty}(\mathbb{T}^{n})$ but non-invertible in $H^{\infty}(\mathbb{D}^{n})$.
\end{theorem}

In particular the backward shift operators $T^{*}_{z_{1}},\ldots,T^{*}_{z_{n}}$ are universal when $n>1$.


\begin{equi3} 
If $U$ is universal for a separable, infnite dimensional Hilbert space $\mathcal{H}$, then the ISP is equivalent to the assertion that every minimal invariant subspace for $U$ is one dimensional.
\end{equi3}



In this work we will use this statement to obtain sufficient conditions for the ISP. More precisely, we will obtain conditions on $\varphi\in H^{\infty}(\mathbb{D})$ to show that $T^{*}_{\varphi}$ on $H^{2}(\mathbb{D}^{2})$ is an universal operator and that no nontrivial invariant subspace of $T^{*}_{\varphi}$ is minimal.

The purpose of this paper is the following. In Section \ref{s1}, we collect some of the preliminaries. In Section \ref{s2}, we introduce the concept of translation generalized inner function. In Section \ref{s3}, we provide sufficient conditions for ISP to be true (Theorem \ref{t6}). For that, we need two result sets (Theorems \ref{t7}, \ref{t8} and \ref{t9}). In Theorems \ref{t7} and \ref{t8}, we consider $\varphi\in H^{\infty}(\mathbb{D})$ and $M\subset H^2(\mathbb{D}^2)$ a $T_{\varphi}^{*}$-invariant subspace and we obtain conditions such that $T_{\varphi}^{*}|_M$ has a proper invariant subspace. In Theorem \ref{t9}, we introduce the operator $J_{g,\varphi}$ so that $M$ does not have a nontrivial invariant subspace.

\section{Preliminaries}\label{s1}

\subsection{The Hardy space $H^{2}(\mathbb{D}^{2})$}

Let $\mathbb{D}$ be the unit disk in the complex plane $\mathbb{C}$ and $\mathbb{T}=\partial\mathbb{D}$ its boundary. The bidisk $\mathbb{D}^{2}$ and 2-torus $\mathbb{T}^{2}$ are the Cartesian product of 2 copies of $\mathbb{D}$ and $\mathbb{T}$ respectively. Let $L^{2}(\mathbb{T}^{2})$ denote the usual Lebesgue space and $L^{\infty}(\mathbb{T}^{2})$ the essentially bounded functions with respect the normalized Haar measure $\sigma$.

The Hardy space on the bidisk $H^{2}(\mathbb{D}^{2})$ is defined as the class of all holomorphic functions $f\in\mathbb{D}^{2}$ for which
\begin{equation*}
\|f\|^{2}=\sup_{0<r<1}\int_{\mathbb{T}^{2}}|f(r\zeta)|^{2}d\sigma(\zeta)<\infty.    
\end{equation*}

It is well known that if $f\in H^{2}(\mathbb{D}^{2})$, then the radial limit
$$
f^{*}(\zeta)=\lim_{r\rightarrow1}f(r\zeta)
$$
exists for almost all $\zeta\in\mathbb{T}^{2}$ and
$$
\lim_{r\rightarrow1}\int_{\mathbb{T}^{2}}|f_{r}-f^{*}|^{2}d\sigma=0,
$$
where  $f_{r}(\zeta)=f(r\zeta)$ for all $\zeta\in\mathbb{T}^{2}$. Thus, $H^{2}(\mathbb{D}^{2})$ can be viewed a bounded subspace of $L^{2}(\mathbb{T}^{2})$.

Denote by $H^{\infty}(\mathbb{D}^{2})$ the space of bounded analytic functions on $\mathbb{D}^{2}$. Also via radial limits, $H^{\infty}(\mathbb{D}^{2})$ can be seen as a subspace of $L^{\infty}(\mathbb{T}^{2})$. An inner function in $\mathbb{D}^{2}$ is a function $f\in H^{\infty}(\mathbb{D}^{2})$ such that $|f^{*}|=1$ a.e. on $\mathbb{T}^{2}$.

Let $P$ denote the orthogonal projection from $L^{2}(\mathbb{T}^{2})$ onto $H^{2}(\mathbb{D}^{2})$. For a function $\varphi\in L^{\infty}(\mathbb{T}^{2})$, the Toeplitz operator $T_{\varphi}$ with symbol $\varphi$ is defined by
$$
T_{\varphi}f=P(\varphi f)
$$
for $f\in H^{2}(\mathbb{D}^{2})$. Then $T_{\varphi}$ is a bounded linear operator on $H^{2}(\mathbb{D}^{2})$ and its adjoint is given by $T^{*}_{\varphi}=T_{\overline{\varphi}}$. Similarly, for a given $\varphi\in L^{\infty}(\mathbb{T})$, the 1-dimensional Toeplitz operator $t_{\varphi}$ with symbol $\varphi$ is the bounded linear operator on $H^{2}(\mathbb{D})$ defined by
$$
t_{\varphi}f=Q(\varphi f)
$$
for $f\in H^{2}(\mathbb{D})$, where $Q$ is the orthogonal projection from $L^{2}(\mathbb{D})$ onto $H^{2}(\mathbb{D})$. 

In this work, we will often decompose the space $H^{2}(\mathbb{D}^{2})$ as follows. Let $H^{2}(z)$ and $H^{2}(w)$ denote the classical Hardy spaces over $\mathbb{D}$ in the variables $z$ and $w$ respectively. Then $H^{2}(\mathbb{D}^{2})$ may be defined as the $H^{2}(z)$-valued Hardy space
\begin{equation*}
H^{2}(\mathbb{D}^{2})=\left\{g(z,w)=\sum_{n=0}^{\infty}g_{n}(z)w^{n}:\sum_{n=0}^{\infty}\|g_{n}\|^{2}_{H^{2}(z)}<\infty\right\}.    
\end{equation*}

Thus, considering $H_{n}=H^{2}(z)w^{n}$ for each $n\in\mathbb{N}$, we have
$$
H^{2}(\mathbb{D}^{2})=\bigoplus_{n=0}^{\infty}H_{n}.
$$

For each $n\in\mathbb{N}$, we denote by $P_{n}$ the orthogonal projection of $H^{2}(\mathbb{D}^{2})$ onto $H_{n}$.

\subsection{Invariant subspaces of $H^{2}(\mathbb{D}^{2})$}

Let $\varphi\in H^{\infty}(\mathbb{D}^{2})$. A closed subspace $M\subset H^{2}(\mathbb{D}^{2})$ is said to be $T_{\varphi}$-invariant if $\varphi M\subset M$. Beurling's theorem states that every nontrivial invariant subspace of $t_{z}$ is of the form $M=\varphi H^{2}(\mathbb{D})$, where $\varphi$ is an inner function in $\mathbb{D}$. Thus $M$ is a cyclic subspace, i.e.
$$
M=\overline{\operatorname{span}\{(t_{z})^{n}\varphi:n\in\mathbb{N}\}}.
$$

Note that Beurling's theorem cannot be naturally extended to multivariable functions. In fact, considering the polynomial ring $\mathcal{R}=\mathbb{C}[z,w]$, Rudin \cite{Rudin} observed that the invariant subspace
$$
[z-w]:=\overline{\{(z-w)p:p\in\mathcal{R}\}}
$$
is not the form $\varphi H^{2}(\mathbb{D}^{2})$ for any inner function $\varphi\in H^{\infty}(\mathbb{D}^{2})$.

Another property that cannot be transferred when working with invariant subspaces and universality on the Hardy space over the polydisk is as follows. $zH^{2}(\mathbb{D})$ is a nontrivial invariant subspace of $t_{z}$ but $t^{*}_{z}$ is not an universal operator for $H^{2}(\mathbb{D})$ since $\operatorname{dim}\operatorname{ker}t_{\overline{z}}=1$. However, over $H^{2}(\mathbb{D}^{2})$ we have the following:

\begin{proposition}
Let $\varphi\in H^{\infty}(\mathbb{D}^{2})$. Then $T^{*}_{\varphi}$ satisfies the Caradus criterion for universality if, and only if, $\varphi H^{2}(\mathbb{D}^{2})$ is a nontrivial invariant subspace of $H^{2}(\mathbb{D}^{2})$.
\end{proposition}
\begin{proof}
Since $\varphi\in H^{\infty}(\mathbb{D}^{2})$ and $T^{*}_{\varphi}$ is surjective, we have $1/\varphi\in L^{\infty}(\mathbb{T}^{2})$. Hence $\varphi H^{2}(\mathbb{D}^{2})$ is an invariant subspace of $H^{2}(\mathbb{D}^{2})$ by \cite[Thm. 2]{Koca}. On the other hand, since $\operatorname{Ker}(T^{*}_{\varphi})=[\varphi H^{2}(\mathbb{D}^{2})]^{\perp}$ has infinite dimension, we have by Ahern and Clark \cite[p. 969]{Ahern} that $\varphi H^{2}(\mathbb{D}^{2})\neq H^{2}(\mathbb{D}^{2})$. Conversely, if $\varphi H^{2}(\mathbb{D}^{2})$ is an invariant subspace of $H^{2}(\mathbb{D}^{2})$ again by \cite[Thm. 2]{Koca} we have that $1/\varphi\in L^{\infty}(\mathbb{T}^{2})$. Now since
$$
1/\varphi\in H^{\infty}(\mathbb{D}^{2})\Rightarrow\varphi H^{2}(\mathbb{D}^{2})=H^{2}(\mathbb{D}^{2})
$$
and $\varphi H^{2}(\mathbb{D}^{2})$ is a nontrivial invariant subspace follow that $\varphi$ is non-invertible in $H^{\infty}(\mathbb{D}^{2})$ and so $T^{*}_{\varphi}$ satisfies the Caradus criterion by Theorem \ref{t1}.
\end{proof}

We end this section with the following notation. Given $\varphi\in H^{\infty}(\mathbb{D})$ and $g\in H^{2}(\mathbb{D}^{2})$, the minimal closed invariant subspace for $T^{*}_{\varphi}$ that contains $g$ will be denoted by
$$
V_{T^{*}_{\varphi},g}:=\overline{\operatorname{span}\{(T^{*}_{\varphi})^{n}g:n\in\mathbb{N}\}}.
$$

\section{Universal translations of $T^{*}_{\varphi}$}\label{s2}

In general, not every translation of an universal operator is an universal operator. However, Cowen and Gallardo-Gutiérrez \cite[Thm. 2]{Cowen} showed that if $U\in\mathcal{L}(\mathcal{H})$ satisfies the Caradus criterion then there is $\epsilon>0$ so that for $|\mu|<\epsilon$, the operator $U+\mu I$ is universal. Particularly for Toeplitz operators over Hardy space $H^{2}(\mathbb{D}^{2})$, let's introduce the concept of a symbol $\varphi$ such that $T_{\varphi}^{*}$ is not universal but $T_{\varphi}^{*}+\lambda I$ is universal for some $\lambda\in\mathbb{C}$.

\begin{definition}
We say that $\varphi\in H^{\infty}(\mathbb{D}^{2})$ is a translation generalized inner function when there are $z_{0}\in\mathbb{D}^{2}$ and $\delta>0$ such that
$$
|\varphi^{*}(z)-\varphi(z_{0})|>\delta \ \text{a.e.}\  z\in\mathbb{T}^{2}.
$$
\end{definition}

\begin{example}
It follows from \cite[Thm. 2.2.10]{Martinez} that every inner function $\varphi\in H^{\infty}(\mathbb{D})$ is a translation generalized inner function. 
\end{example}


The following fact will be useful in Theorem \ref{t6}.

\begin{proposition}\label{t11}
If $\varphi\in H^{2}(z)$ is an inner function, then $T^{*}_{\varphi}$ has an universal translation for $H^{2}(\mathbb{D}^{2})$. 
\end{proposition}
\begin{proof}
Since $\varphi$ is an inner function on $\mathbb{D}$, we have there are $z_{0}\in\mathbb{D}$ and $\delta>0$ such that
$$
|\varphi^{*}(z)-\varphi(z_{0})|>\delta \ \text{a.e.}\  z\in\mathbb{T}.
$$ 
Thus, considering $\psi(z)=\varphi(z)-\varphi(z_{0})$, it follows from Theorem \ref{t1} that $T^{*}_{\psi}=T^{*}_{\varphi}-\overline{\varphi({z_{0})}}I$ is an universal operator for $H^{2}(\mathbb{D}^{2})$.
\end{proof}

\section{Sufficient conditions for the ISP}\label{s3}

We begin this section by showing that each $H_n$ is a reducing subspace of $T_{\varphi}$.

\begin{proposition}\label{t3}
Let $\varphi \in H^{\infty}(z)$. Then $T_{\varphi}(fw^n)=(t_{\varphi}f)w^n$ and $ T_{\varphi}^{*}(fw^n)=(t_{\varphi}^{*}f)w^n$ for all $f\in H^{2}(z)$ and $n\in \mathbb{N}$. 
\end{proposition}
\begin{proof}
Let $f\in H^{2}(z)$ and $n\in \mathbb{N}$. Since $\varphi \in H^{\infty}(z)$, we have 
$$
T_{\varphi}(fw^n)=\varphi f w^n=(t_{\varphi}f)w^n.
$$
Now, for $g(z,w)=\sum_{n=0}^{\infty}g_n(z)w^n\in H^2(\mathbb{D}^2)$ it follows that
$$
\langle T_{\varphi}g,fw^n\rangle=\langle \varphi g_n,f\rangle=\langle g_n,t_{\varphi}^{*}f\rangle=\langle g,(t_{\varphi}^{*}f)w^n\rangle,
$$
as desired.
\end{proof}

In fact, a bit more is true:
\begin{remark}\label{t2}
If $\varphi\in H^{\infty}(z)$, then 
$$
T_{\varphi}g=\sum_{n=0}^{\infty}(t_{\varphi}g_n)w^n \  \text{and} \ T_{\varphi}^{*}g=\sum_{n=0}^{\infty}(t_{\varphi}^{*}g_n)w^n
$$
for all $g(z,w)=\sum_{n=0}^{\infty}g_{n}(z)w^n\in H^2(\mathbb{D}^2)$.
\end{remark}

Let $\varphi\in H^{\infty}(\mathbb{D})$ and $M\subset H^2(\mathbb{D}^2)$ be a $T_{\varphi}^{*}$-invariant subspace. Our initial goal is to obtain conditions such that $T^{*}_{\varphi}|_{M}$ has a nontrivial invariant subspace. This will be provided with Theorems \ref{t7}, \ref{t8} and \ref{t9}.

\begin{lemma}\label{t5}
If $\varphi\in H^{\infty}(z)$, then 
$$
P_n T_{\varphi}^{*}=T_{\varphi}^{*} P_n
$$
for all $n\in\mathbb{N}$.
\end{lemma}
\begin{proof}
In fact, for $g(z,w)=\sum_{m=0}^{\infty}g_m(z)w^m\in H^2(\mathbb{D}^2)$, we have by Remark \ref{t2} and Proposition \ref{t3}, respectively, that
$$
P_n T_{\varphi}^{*}(g)=P_n\left(\sum_{m=0}^{\infty} (t_{\varphi}^{*}g_m)w^m\right)= (t_{\varphi}^{*}g_n)w^n
$$
and
$$
T_{\varphi}^{*} P_n(g)= T_{\varphi}^{*}(g_nw^n)=(t_{\varphi}^{*}g_n)w^n.
$$
\end{proof}

\begin{lemma}\label{t4}
Let $\varphi\in H^{\infty}(z)$. If $M\subset H^2(\mathbb{D}^2)$ is an invariant subspace of $T_{\varphi}^{*}$, then there is an invariant subspace $V\subset H^2(\mathbb{D})$ of $t_{\varphi}^{*}$ such that $P_n(M)$ is dense in $Vw^n$.
\end{lemma}
\begin{proof}
Consider $V_0=\{g \in H^{2}(\mathbb{D}): gw^n \in P_n(M)\}$. It is easy to see that $V_{0}$ is a closed subspace of $H^{2}(\mathbb{D})$. If $g\in V_{0}$, then there is $f\in M$ such that $gw^{n}=P_{n}f$. Thus
$$
(t^{*}_{\varphi}g)w^{n}=(t^{*}_{\varphi}P_{n})f=P_{n}(t^{*}_{\varphi}f).
$$
Since $M$ is $T^{*}_{\varphi}$-invariant, we have $(t_{\varphi}^{*}g)w^n\in P_n(M)$ and therefore $V_{0}$ is $t_{\varphi}^{*}$-invariant. Considering $V=\overline{V_0}$ we have that $V$ is $t_{\varphi}^{*}$-invariant and $P_n(M)$ is dense in $Vw^n$.
\end{proof}

The first case in which we obtain a nontrivial invariant subspace of $T_{\varphi}^{*}|_M$ is when one of the $P_{n}(M)$ has finite dimension.

\begin{theorem}\label{t7}
Let $\varphi\in H^{\infty}(\mathbb{D})$ and $M\subset H^2(\mathbb{D}^2)$ be a $T_{\varphi}^{*}$-invariant subspace. If there exists $n\in \mathbb{N}$ such that $\operatorname{dim}P_n(M)<\infty$, then $T_{\varphi}^{*}|_M$ has a nontrivial invariant subspace.
\end{theorem}
\begin{proof}
Since $\operatorname{dim}P_n(M)<\infty$, it follows from Lemma \ref{t4} that $P_n(M)$ is $t_{\varphi}^{*}$-invariant. Moreover there is $g\in M$ such that $P_n(g)=uw^n$, where $t_{\varphi}^{*}u=\lambda u$. Therefore, by Lemma \ref{t5}, $P_n(h)=0$ where $h=T_{\varphi}^{*}g-\lambda g$. So either $g$ is an eigenvector of $T_{\varphi}^{*}$ or $V_{T^{*}_{\varphi},h}\neq M$. In either case, $T_{\varphi}^{*}|_M$ has an invariant subspace.
\end{proof}

In light of the previous theorem, we observe that there are many invariant subspaces $M$ that do not satisfy such assumptions. Indeed, let $g\in H^2{(\mathbb{D})}$ such that
$\operatorname{dim}V_{t^{*}_{\varphi},g}=\infty$. Consider then $f\in H^2(\mathbb{D}^2)$ given by $P_n(f)=\lambda^n gw^n$, where $|\lambda|<1$. Thus if $M=V_{T^{*}_{\varphi},f}$ we have that $\operatorname{dim}P_n(M)=\infty$ for all $n\in \mathbb{N}$ and moreover $M\neq H^2(\mathbb{D}^2)$.

We must then consider a second version of Theorem \ref{t7}.

\begin{theorem}\label{t8}
Let $\varphi\in H^{\infty}(\mathbb{D})$ and $M\subset H^2(\mathbb{D}^2)$ be a $T_{\varphi}^{*}$-invariant subspace such that $\operatorname{dim}P_n(M)=\infty$. If $t_{\varphi}^{*}|_V$ has a invariant subspace for all subspace invariant $V \subset H^2(\mathbb{D})$ and $P_n(M) $ is closed for some $n\in \mathbb{N}$, then $T_{\varphi}^{*}|_M$ has an invariant subspace.
\end{theorem}
\begin{proof}
Since $P_n(M)$ is closed, it follows from Lemma \ref{t4} that $P_n(M)=Vw^n$, where $V\subset H^2(\mathbb{D})$ is a $t_{\varphi}^{*}$-invariant. 
Let then be $V_0\subset V$ an invariant subspace of $t_{\varphi}^{*}$ and consider 
$$
U=\overline{\operatorname{span}\lbrace v_1+v_2:v_1\in V_0 \mbox{ and } v_2\in H_{n}^{\perp} \rbrace}.
$$

Thus, we have that $U$ is an invariant subspace of $T_{\varphi}^{*}$ with $\{0\}\neq M\cap U\neq M$, i.e. $T_{\varphi}^{*}|_M$ has an invariant subspace. 
\end{proof}

Now, let's consider a more general case than Theorems \ref{t7} and \ref{t8}. First, we need the following lemma.

\begin{lemma}\label{t10}
Let $\varphi\in H^2(\mathbb{D})$ an inner function. If $M\subset H^2(\mathbb{D}^2)$ is an invariant subspace of $T_{\varphi}^{*}$ for all $g\in M$, then
$$
\left\{\sum_{i=0}^{\infty}\beta_{i}((t_{\varphi}^{*})^ig_n)w^n:\beta=(\beta_0,\beta_1,...)\in\ell^2(\mathbb{C})\right\}\subset P_n(M).
$$
\end{lemma}
\begin{proof}
Since $\varphi$ is an inner function, we have that $\Arrowvert T_{\varphi}^{*}\Arrowvert=1$ and so $\Arrowvert (T_{\varphi}^{*})^n g\Arrowvert\leq \Arrowvert g\Arrowvert$. Thus if $\beta =(\beta_0,\beta_1,...)\in\ell^2(\mathbb{C})$, then $\sum_{i=0}^{\infty} \beta_i(T_{\varphi}^{*})^ig\in M $ and therefore $\sum_{i=0}^{\infty}\beta_i((t_{\varphi}^{*})^ig_n)w^n\in P_n(M) $.
\end{proof}

It follows from the previous lemma that, for each $g\in M$, the operator $J_{g_{n},\varphi}:\ell^{2}(\mathbb{C})\rightarrow\overline{P_{n}(M)}$ given by 
$$
J_{g,\varphi}(\beta)=\sum_{i=0}^{\infty}\beta_i((t_{\varphi}^{*})^ig_n)w^n
$$
is well defined.

Note that if $T_{\varphi}^{*}$ has a minimal invariant subspace $M$, then $M=V_{T^{*}_{\varphi},g}$ for all $g\in M$ and $ J_{g,\varphi}(\ell^2(\mathbb{C}))\cap U=\lbrace 0\rbrace$, for all $g\in M$ and $U\subset \overline{P_n(M)}$ invariant to $t_{\varphi}^{*}$. Now we can consider the following:

\begin{theorem}\label{t9}
Let $\varphi\in H^{\infty}(\mathbb{D})$ an inner function. If for each $g\in H^2(\mathbb{D})$, there is an invariant subspace $U \subset V_{t^{*}_{\varphi},g}$ such that $ J_{g,\varphi}(l^2(\mathbb{C}))\cap U \neq \lbrace 0\rbrace $, then $M\subset H^2(\mathbb{D}^2)$ is not a minimal invariant subspace.
\end{theorem}
\begin{proof}
Let $n\in \mathbb{N}$ such that $P_n(M)\neq \lbrace 0\rbrace$. 

If $\operatorname{dim}P_n(M)<\infty$, then Theorem \ref{t7} guarantees that $T^{*}_{\varphi}|_{M}$ has an invariant subspace and so $M$ is non minimal.

Consider then $\operatorname{dim}P_n(M)=\infty$. Given $g\in H^{\infty}(\mathbb{D})$, consider the invariant subspace $U\subset V_{t^{*}_{\varphi},g}$ such that $ J_{g,\varphi}(l^2(\mathbb{C}))\cap U \neq \lbrace 0\rbrace$. Since $\operatorname{dim}P_n(M)=\infty$, then Lemma \ref{t10} guarantees that $U\subset \overline{P_n(M)}$ and $g\in P_n(M)$. Thus $P_n(M)\cap U\neq \lbrace 0\rbrace$ and so 
$$
M\cap \overline{\operatorname{span}\lbrace u+v:u\in U \mbox{ and } v\in H_{n}^{\perp}\rbrace}\neq \lbrace 0\rbrace
$$
therefore $M$ is non minimal.
\end{proof}

The next result provides us with sufficient conditions for the ISP to be true.

\begin{theorem}\label{t6}
If there exists an inner function $\varphi\in H^2(\mathbb{D})$ such that for each  $g\in H^2(\mathbb{D})$, there is an invariant subspace $U\subset V_{t^{*}_{\varphi},g}$ of $t_{\varphi}^{*}$ so that $U\cap J_{g,\varphi}(l^2(\mathbb{C}))\neq \lbrace 0\rbrace$, then the ISP is true.
\end{theorem}
\begin{proof}
Since $\varphi$ is an inner function, we have by Proposition \ref{t11} that $T^{*}_{\varphi}$ has an universal translation for $H^{2}(\mathbb{D}^{2})$. 
On the other hand, it follows from Theorem \ref{t9} that the invariant subspace $M\subset H^2(\mathbb{D}^2)$ of $T_{\varphi}^{*}$ is not minimal. Thus the ISP is true.
\end{proof}







\end{document}